\documentclass{amsart}
\usepackage{color}
\usepackage{graphicx} 
\usepackage{float}
\usepackage[letterpaper,top=2cm,bottom=2cm,left=3cm,right=3cm,marginparwidth=1.75cm]{geometry}

\newtheorem{theorem}{Theorem}[section]
\newtheorem{lemma}[theorem]{Lemma}
\newtheorem{proposition}[theorem]{Proposition}
\newtheorem{corollary}[theorem]{Corollary}

\theoremstyle{definition}
\newtheorem{definition}[theorem]{Definition}
\newtheorem{example}[theorem]{Example}

\theoremstyle{remark}
\newtheorem{remark}[theorem]{Remark}

\numberwithin{equation}{section}

\def\CC{\mathbb C}
\def\RR{\mathbb R}
\def\NN{\mathbb N}
\def\TT{\mathbb T}
\def\e{\varepsilon}
\def\Re{{\rm Re}\,}
\def\Im{{\rm Im}\,}
\def\cR{{\mathcal R}}
\def\la{\lambda}
\def\ind{{\rm ind}}

\def\mapright#1{\smash{\mathop{\longrightarrow}\limits^{#1}}}
\def\mapdown#1{\Big\downarrow\rlap{$\vcenter{\hbox{$\scriptstyle#1$}}$}}

\begin{document}
\title[Real linear operators and numerical ranges]{Real linear operators and numerical ranges}
\author{Damian Ko\l aczek}
\address{Department of Applied Mathematics, University of Agriculture, ul. Balicka 253c, 30-198 Krak{\'o}w, Poland}
\email{Damian.Kolaczek@urk.edu.pl}

\author{Vladimir M\"uller}
\address{Institute of Mathematics,
Czech Academy of Sciences,
\v Zitna 25, Prague,
 Czech Republic}
\email{muller@math.cas.cz}


%\date{\today}

\thanks{The second author was supported by grant no. 25-15444K of GA \v CR and RVO:67985840}


\begin{abstract}
Real linear operators between two complex Banach spaces unify naturally two important classes of linear operators and antilinear operators. We give a~survey of basic geometric, spectral and duality properties of real linear operators.

The main goal of the paper is to introduce the numerical range of real linear operators on both Hilbert and Banach spaces and to study its properties. In particular, we show that the numerical range of real linear operators on complex Hilbert space is always a~convex set (on at least two-dimensional spaces). This generalizes the classical result of Hausdorff and Toeplitz for linear operators. 
\end{abstract}

\maketitle

\section{Introduction}

Let $X,Y$ be complex Banach spaces. A continuous mapping $\Phi:X\to Y$ is called a~real linear operator if $\Phi(x+x')=\Phi(x)+\Phi(x')$ and $\Phi(tx)=t\Phi(x)$ for all $x,x'\in X$ and $t\in\RR$. Real linear operators appear frequently in various applications, see e.g. \cite{GP}, \cite{GPP}, \cite{IM}, \cite{AP}, \cite{PS} and the references therein. A survey of the properties of real linear operators on complex Hilbert spaces was given in \cite{HR}.

The most important examples of real linear operators are linear operators and antilinear operators.
A continuous mapping $\Phi:X\to Y$ is called an antilinear operator if $\Phi(x+x')=\Phi(x)+\Phi(x')$ and $\Phi(\lambda x)=\bar\la \Phi(x)$ for all $x,x'\in X$ and $\la\in\CC$. 

Some properties of antilinear operators are analogous to the properties of linear operators, but there are also some significant differences. For various aspects of antilinear operators, see e.g. \cite{U}, \cite{HP0}, \cite{CKLP}, \cite{CKLP1}, \cite{KLL}, \cite{KLL1}, \cite{PSW}.

Each real linear operator can be written in a~canonical way as a~sum of a~linear operator and an antilinear operator.

In the present paper, we continue the study of real linear operators.
The paper is organized as follows. In the second section, we summarize some basic facts concerning real linear operators. In the next section, we show that the linear operators and antilinear operators are Birkhoff-James orthogonal in the Banach space of all real linear operators (between two given Banach spaces).

The set of all real linear operators on a~Banach space is a~real Banach algebra. The spectrum of real linear operators on complex Hilbert spaces was introduced and studied in \cite{HR} and \cite{T}.  
The most important difference from the spectral theory of linear operators is that the spectrum
may be empty. In the fourth section of the present paper, we summarize the basic properties of spectrum of real linear operators on Banach spaces.

The main goal of the paper is to introduce and study the notion of numerical range for real linear operators in the general case. For the special case of real linear operators on Hilbert space $\CC^n$ the notion of numerical range was already introduced in~\cite{HN}. 
The numerical range of linear operators is already a~classical concept and plays an important role in operator theory. The standard reference are two manuscripts of Bonsall and Duncan \cite{BD}, \cite{BD1}. For a more recent survey, see \cite{GR}.

The numerical range was recently extended and studied in the antilinear setting by a~number of papers, see \cite{ChHL1}, \cite{HL}, \cite{KM}.

In the present paper, we introduce and study the numerical range of real linear operators both on Hilbert and on Banach spaces. We prove the basic properties of the numerical range and its relation with the spectrum and duality.

Not surprisingly, the numerical range has nicer properties in the Hilbert space setting.  
We show that if the complex Hilbert space is at least two-dimensional then the numerical range of a~real linear operator is always convex subset of the complex plane. This generalizes the famous Hausdorff-Toeplitz theorem for linear operators.

It is well known that the numerical radius $w(T)$ of a~linear operator $T$ satisfies the inequality $w(T)\geq \|T\|/2.$
If $T$ is self-adjoint, then $w(T)=\|T\|$. For antilinear and real linear operators, the situation is more complicated. We study to what extent such results can be generalized to the real linear setting.

Another important result in the theory of the numerical range of linear operators is that the numerical range of a~$2\times 2$-matrix is always an elliptical disc (possibly degenerated to a~line segment or a~single point). We show that this is no longer true for real linear operators. The numerical range of a~real linear operator on a~$2$-dimensional complex Hilbert space is always convex, but it can be a~more complicated set.

In the last section, we give some examples of this kind obtained numerically by a~computer.

\section{Basic results}
Let $X,Y$ be complex Banach spaces. 
A mapping $\Phi:X\to Y$ is called real linear operator if it is continuous, $\Phi(x+x')=\Phi(x)+\Phi(x')$ and $\Phi(tx)=t\Phi(x)$ for all $x,x'\in X$ and real $t$. It is well known that a~real linear mapping is continuous if and only if it is bounded.  

For a~real linear operator $\Phi:X\to Y$ we will write $\|\Phi\|=\sup\{\|\Phi(x)\|: x\in X, \|x\|\le 1\}$. Denote by $\cR(X,Y)$ the set of all continuous real linear operators from $X$ to $Y$. If $Y=X$, then we write shortly $\cR(X)=\cR(X,X)$. 

In fact, a~continuous additive mapping is automatically real linear. This observation is already classical, see \cite {B}.

\begin{theorem}
Let $\Phi:X\to Y$ be a~continuous additive mapping (i.e., $\Phi(x+x')=\Phi(x)+\Phi(x')$ for all $x,x'\in X$). Then $\Phi(tx)=t\Phi(x)$ for all $x\in X$ and real $t$.
\end{theorem}

\begin{proof}
It is easy to see that $\Phi(nx)=n\Phi(x)$ for all $x\in X$ and $n\in\NN$. Consequently, $\Phi(tx)=t\Phi(x)$ for all $x\in X$ and rational $t$. Using the continuity of $\Phi$, we get $\Phi(tx)=t\Phi(x)$ for all $x\in X$ and the real number $t$.
\end{proof}

Moreover, any bounded additive mapping is continuous at the origin, and so continuous, see \cite {B}.

Basic examples of real linear operators are linear operators and antilinear operators.
In fact, each real linear operator can be written as a~sum of a~linear operator and an antilinear operator, cf. \cite{HR}.

\begin{theorem}
Let $\Phi:X\to Y$ be a~real linear operator. Then there exist a~linear operator $T:X\to Y$ and an antilinear operator $A:X\to Y$ such that $\Phi=T+A$. The operators $T$ and $A$ are determined uniquely.
\end{theorem}

\begin{proof}
Define $T$ and $A$ by
$$
Tx=\frac{1}{2}\bigl(\Phi (x)-i\Phi(ix)\bigr)
$$
and 
$$
Ax=\frac{1}{2}\bigl(\Phi (x)+i\Phi(ix)\bigr)
$$
for all $x\in X$. It is easy to see that $\Phi=T+A$, $T$ is a linear operator and $A$ an antilinear operator.

If $\Phi=T'+A'$ for another pair $T', A'$ with $T'$ linear and $A'$ antilinear, then
$(T-T')+(A-A')=0$. So for each $x\in X$ we have $(T-T')x+(A-A')x=0$ and 
$$
0=(T-T')(ix)+(A-A')(ix)= i(T-T')x-i(A-A')x,
$$
and so $(T-T')x-(A-A')x=0$. Hence $(T-T')x=0$ and $(A-A')x=0$. Since $x\in X$ was arbitrary, we have $T'=T$ and $A'=A$.

\end{proof}

Denote by $X^*$ and $Y^*$ the dual spaces of $X$ and $Y$, respectively. For a~linear operator $T:X\to Y$ let $T^*:Y^*\to X^*$ be its dual, i.e., the linear operator satisfying $\langle Tx,y^*\rangle=\langle x,T^*y^*\rangle$ for all $x\in X$ and $y^*\in Y^*$. 

For an antilinear operator $A:X\to Y$, let $A^*:Y^*\to X^*$ be its antilinear dual, i.e., the antilinear operator satisfying $\langle Ax,y^*\rangle=\overline{\langle x,A^*y^*\rangle}$ for all $x\in X$ and $y^*\in Y^*$. It is natural to define the dual of a~real linear operator using its canonical linear/antilinear decomposition.

If $\Phi:X\to Y$ is a~real linear operator, $\Phi=T+A$ with $T,A:X\to Y$, $T$ linear and $A$ antilinear, then define
$\Phi^*:Y^*\to X^*$ by $\Phi^*=T^*+A^*$.

Clearly
$$
\Re\langle \Phi x,y^*\rangle=\Re\langle x,\Phi^* y^*\rangle
$$ 
for all $x\in X$ and $y^*\in Y^*$.

The duality satisfies the usual properties.

\begin{proposition}
Let $\Phi,\Psi:X\to Y$ be real linear operators. Then
$(\Phi+\Psi)^*=\Phi^*+\Psi^*$.
\end{proposition}

\begin{proposition} Let $X,Y,Z$ be Banach spaces. 
Let $\Phi:X\to Y$ and $\Psi:Y\to Z$ be real linear operators. Then
$(\Psi\Phi)^*=\Phi^*\Psi^*$.
\end{proposition}

\begin{proof}
Let $\Phi=T+A$, $\Psi=S+B$ be the canonical linear/antilinear decompositions of $\Phi$ and $\Psi$, respectively. Note that the product of two antilinear operators is linear and the product of a~linear operator and an antilinear operator is antilinear. So
$$
\Psi\Phi= (ST+BA)+(SA+TB)
$$
where $ST+BA$ is a~linear operator and $SA+TB$ antilinear.
By definition, 
$$
(\Psi\Phi)^*= (ST+BA)^*+(SA+TB)^*= (ST)^*+(BA)^*+(SA)^*+(BT)^*.
$$
On the other hand,
$$
\Phi^*\Psi^*=(T+A)^*(S+B)^*=
T^*S^*+T^*B^*+A^*S^*+A^*B^*.
$$
Clearly $(ST)^*=T^*S^*$. So it is sufficient to show that
\begin{itemize}
\item[(i)] $(SA)^*=A^*S^*$;
\item[(ii)] $(BT)^*=T^*B^*$;
\item[(iii)] $(BA)^*=A^*B^*$.
\end{itemize}

For $x\in X$ and $z^*\in Z^*$ we have 
$$
\langle SAx,z^*\rangle=
\langle Ax,S^*z^*\rangle=
\overline{\langle x, A^*S^*z^*\rangle}.
$$
So $(SA)^*=A^*S^*$, which proves (i). The remaining equalities (ii) and (iii) can be proved similarly.
\end{proof}

\begin{remark}\label{R2.5}
Let $X$ be a~complex Banach space and $X^*$ its dual, i.e., the set of all continuous complex functionals. Let $X_r$ be the space $X$ considered as a~real Banach space and $X_r^*$ its dual, that is, the set of all continuous real functionals on $X_r$. 

The spaces $X^*$ and $X_r^*$ are closely related. Let $x^*\in X^*$. Then $\Re x^*\in X_r^*$ (where $\Re x^*$ is defined by
$\langle x,\Re x^*\rangle=\Re\langle x,x^*\rangle$ for all $x\in X=X_r$). The mapping $x^*\mapsto \Re x^*$ is real linear, isometric and bijective, see \cite{BD}, Lemma 15.3. Let $\eta_X:X^*\to X_r^*$ be defined by $\eta_X(x^*)=\Re x^*$. Note that $\eta_X^{-1}$ satisfies
$\langle x,\eta_X^{-1}x_r^*\rangle= \langle x,x_r^*\rangle-i\langle ix,x_r^*\rangle$ for all $x\in X$ and $x_r^*\in X_r^*$.

Let $X,Y$ be complex Banach spaces and $\Phi:X\to Y$ a~real linear operator. Clearly $\Phi$ may be considered also as a~real linear operator $\Phi_r:X_r\to Y_r$. In addition to the dual $\Phi^*:Y^*\to X^*$ defined above, we may also consider the dual $\Phi_r^*:Y_r^*\to X_r^*$ defined in the usual way by $\langle x, \Phi_r^* y_r^*\rangle=\langle \Phi_r x,y_r^*\rangle$ for all $x\in X_r=X, y_r^*\in Y_r^*$. The operators $\Phi^*$ and $\Phi_r^*$ are related in the following way.

For $x\in X$ and $y^*\in Y^*$ we have
$$
\langle x, \eta_X(\Phi^*y^*)\rangle=
\Re\langle x, \Phi^*y^*\rangle=
\Re\langle \Phi x, y^*\rangle=
\langle \Phi x,\eta_Y(y^*)\rangle=
\langle x, \Phi_r^*\eta_Y(y^*)\rangle.
$$
So $\eta_X \Phi^*=\Phi_r^*\eta_Y$. Equivalently, $\Re \Phi^* y^*= \Phi_r^*(\Re y^*)$ for all $y^*\in Y^*$. Thus the following diagram commutes:
$$
\begin{matrix}
Y^*&\mapright{\eta_Y}&Y_r^*\cr
\mapdown{\Phi^*}&&\mapdown{\Phi_r^*}\cr
X^*&\mapright{\eta_X}&X_r^*
\end{matrix}
$$
\end{remark}

Recall that $\eta_X$ and $\eta_Y$ are bijective isometries. So many well-known properties for $\Phi_r^*$ are automatically satisfied also for $\Phi^*$.

\begin{proposition}\label{P2.6}
Let $X$ be a~complex Banach space and $\Phi:X\to X$ a~real linear operator. Then:
\begin{itemize}
\item[(i)] $\Phi$ has closed range if and only if $\Phi^*$ has closed range;
\item[(ii)] $\Phi^*$ is injective if and only if $\Phi$ has dense range;
\item[(iii)] $\Phi$ is injective if and only if $\Phi^*$ has $w^*$-dense range.
\end{itemize}
\end{proposition}

\section{Birkhoff-James orthogonality}

Recall that a~vector $x$ in a~Banach space $Z$ is Birkhoff-James orthogonal to a~vector $y\in Z$ (in symbol $x\perp_B y$) if
$$
\|x+\lambda y\|\ge\|x\|
$$
for all $\lambda\in\CC$.
The Birkhoff-James orthogonality is a~natural generalization of the orthogonality in Hilbert spaces. However, in general, it is not a~symmetrical relation.

The symmetrized version of the Birkhoff-James orthogonality is a~much stronger relation. For details, see e.g. \cite{AGKRZ}.

Let $X,Y$ be Banach spaces. Clearly, the set $\cR(X,Y)$ of all real linear operators from $X$ to $Y$ with the norm $\|\cdot\|$ is a~Banach space. 
We show that for any linear operator $T:X\to Y$ and antilinear operator $A:X\to Y$ we have both $T\perp_B A$ and $A\perp_B T$.

\begin{proposition}
Let $X$, $Y$ be Banach spaces, let $T:X\to Y$ be a~linear operator, $A:X\to Y$ an antilinear operator, 
let $\alpha,\beta,\alpha',\beta'\in\CC$, $|\alpha|\le|\alpha'|, |\beta|\le|\beta'|$. Then
$$
\|\alpha T+\beta A\|\le\|\alpha' T+\beta' A\|.
$$
\end{proposition}

\begin{proof}
Let $x\in X$ and $y^*\in Y^*$ be unit vectors. We have
\begin{align*}
&\max_{0\le\theta<2\pi}\bigl|\langle(\alpha T+ \beta A)e^{i\theta} x, y^*\rangle\bigr|
\\
=
&\max_{0\le\theta<2\pi}\Bigl(\bigl|(e^{i\theta}\alpha \langle Tx,y^*\rangle+ e^{-i\theta}\beta\langle Ax,y^*\rangle\bigr|\Bigr)
\\
=
&\max_{0\le\theta<2\pi}\Bigl(\bigl|\alpha \langle Tx,y^*\rangle+ e^{-2i\theta}\beta\langle Ax,y^*\rangle\bigr|\Bigr)
\\
=&|\alpha|\cdot|\langle Tx,y^*\rangle|+|\beta|\cdot|\langle Ax,y^*\rangle|.
\end{align*}
So
\begin{align*}
&\|\alpha T+\beta A\|\\
&\sup\Bigl\{\max_{0\le\theta<2\pi}\{|\langle\alpha T+\beta A)e^{i\theta}x,y^*\rangle|\}: x\in X, y^*\in Y^*, \|x\|=1=\|y^*\|\Bigr\}
\\
=&\sup\Bigl\{|\alpha|\cdot |\langle Tx,y^*\rangle|+|\beta|\cdot|\langle Ax,y^*\rangle|: x\in X, y^*\in Y^*, \|x\|=1=\|y^*\|\Bigr\}
\\
\le
&\sup\Bigl\{|\alpha'|\cdot |\langle Tx,y^*\rangle|+|\beta'|\cdot|\langle Ax,y^*\rangle|: x\in X, y^*\in Y^*, \|x\|=1=\|y^*\|\Bigr\}
\\
=
&\sup\Bigl\{\max_{0\le\theta<2\pi}|\langle(\alpha' T+\beta' A)e^{i\theta}x,y^*\rangle|:x\in X, y^*\in Y^*, \|x\|=1=\|y^*\|\Bigr\}
\\
=&\|\alpha' T+\beta'A\|.
\end{align*}
\end{proof}

\begin{corollary}
Let $X,Y$ be Banach spaces, let $T:X\to Y$ be a~linear operator and $A:X\to Y$ an antilinear operator. 
Then
$$
T\perp_B A \qquad{\rm and}\qquad A\perp_B T
$$
in $R(X,Y)$.
\end{corollary}

\section{Spectral theory}

Let $X$ be a~complex Banach space. Then the set $\cR(X)$ of all real linear operators is a~complex Banach space. In the same time, it is a~real Banach algebra (the multiplication in $\cR(X)$ is associative and distributive, however, the axiom $\alpha(\Phi\Psi)=(\alpha\Phi)\Psi=\Phi(\alpha\Psi)$ is true only for $\Phi,\Psi\in \cR(X)$ and real $\alpha$; for complex $\alpha$ it is not satisfied in general).

We define the spectrum of operators $\Phi\in \cR(X)$ in the natural way:

\begin{definition}
Let $\Phi\in \cR(X)$. The spectrum $\sigma(\Phi)$ is defined by
$$
\sigma(\Phi)=\{\lambda\in\CC: \Phi-\lambda I \hbox{ is not bijective}\},
$$
where $I$ denotes the identity operator on $X$.
\end{definition}

Note that if $\Phi-\lambda I$ is bijective, then $(\Phi-\lambda I)^{-1}$ is continuous by the Banach open mapping theorem for real Banach spaces.

The spectral theory for real linear operators on complex Hilbert spaces was studied in \cite{HR}, see also \cite{T}.
Many properties of the spectrum of real linear operators are analogous to those of complex linear operators. However, there is a~significant difference --- the spectrum of a~real linear operator may be empty. The simplest example of this kind was given in \cite{HR}.
The spectrum of the real linear operator $(z_1,z_2)\mapsto (\bar z_2,-\bar z_1)$ acting on $\CC^2$ is empty.

On the other hand, the spectrum of a~real linear operator may contain uncountably many elements even in finite-dimensional spaces, see \cite{HP}.

We also define the following parts of the spectrum:

\begin{definition}
Let $\Phi\in \cR(X)$. Define:
\begin{itemize}
\item[(i)] point spectrum 
$$\sigma_p(\Phi)=\{\lambda\in\CC: \Phi-\lambda I \hbox{ is not injective}\};$$
\item[(ii)] approximate point spectrum 
$$\sigma_{ap}(\Phi)=\{\lambda\in\CC: \Phi-\lambda I \hbox{ is not bounded below}\};$$
\item[(iii)] surjective spectrum 
$$\sigma_{sur}(\Phi)=\{\lambda\in\CC: \Phi-\lambda I \hbox{ is not surjective}\};$$
\end{itemize}

\end{definition}

The basic properties of the spectrum of real linear operators are given in the following theorem.

\begin{theorem}\label{T4.3}
Let $X$ be a~complex Banach space and $\Phi\in \cR(X)$. Then:
\begin{itemize}
\item[(i)] \begin{align*}
&\sigma_p(\Phi)\subset\sigma_{ap}(\Phi)\subset\sigma(\Phi),\cr
&\sigma_{sur}(\Phi)\subset\sigma(\Phi),\cr
&\sigma(\Phi)=\sigma_p(\Phi)\cup\sigma_{sur}(\Phi)=\sigma_{ap}(\Phi)\cup\sigma_{sur}(\Phi);
\end{align*}
\item[(ii)] $\sigma(\Phi)\subset\{z\in \CC:|z|\le\|\Phi\|\}$;
\item[(iii)] $\sigma(\Phi)$, $\sigma_{ap}(\Phi)$ and $\sigma_{sur}(\Phi)$ are compact subsets of $\CC$;
\item[(iv)] 
\begin{align*}
&\sigma_{ap}(\Phi)=\sigma_{sur}(\Phi^*),\cr
&\sigma_{sur}(\Phi)=\sigma_{ap}(\Phi^*),\cr
&\sigma(\Phi)=\sigma(\Phi^*);
\end{align*}
\item[(v)] $\partial\sigma(\Phi)\subset \sigma_{ap}(\Phi)\cap\sigma_{sur}(\Phi)$, where $\partial$ denotes the topological boundary.

\end{itemize}

\end{theorem}

\begin{proof}
$(i)$ is obvious.
\smallskip

(iv) follows from Proposition \ref{P2.6}.
\smallskip

(ii): Let $|\lambda|>\|\Phi\|$. Then the series $\sum_{j=0}^\infty \frac{-\Phi^j}{\lambda^{j+1}}$ is convergent and
is equal to
$(\Phi-\lambda I)^{-1}$. So $\sigma(\Phi)\subset\{z\in \CC:|z|\le\|\Phi\|\}$.
\smallskip

(iii): Clearly $\lambda\in\sigma_{ap}(\Phi)$ if and only if $\inf\bigl\{\|(\Phi-\lambda I)x\|: x\in X, \|x\|=1\bigr\}=0$.
It is easy to see that $\sigma_{ap}(\Phi)$ is a~closed set.

By (iv), $\sigma_{sur}(\Phi)=\sigma_{ap}(\Phi^*)$, so $\sigma_{sur}(\Phi)$ is also closed.
Hence the spectrum $\sigma(\Phi)=\sigma_{ap}(\Phi)\cup\sigma_{sur}(\Phi)$ is closed, and so a~compact subset of $\CC$. 
\smallskip

(v): Let $\lambda\in\partial\sigma(\Phi)$. 
Without loss of generality, we may assume that $\lambda=0$.
Then there exists a~sequence $(\lambda_n)\subset\CC\setminus\sigma(\Phi)$ such that $\lambda_n\to 0$.

We show that $\|(\Phi-\lambda_n I)^{-1}\|\to\infty$. Suppose the contrary. We have
$$
\Phi(\Phi-\lambda_n I)^{-1}=I+\lambda_n(\Phi-\lambda_n I)^{-1}.
$$
Let 
$$
\Psi_n=\lambda_n(\Phi-\lambda_n I)^{-1}.
$$
Choose $n$ sufficiently large such that
$\|\Psi_n\|<1.$

Then the series $\sum_{j=0}^\infty(-1)^j\Psi_n^j$ is convergent and converges to the inverse of $I+\Psi_n=\Phi(\Phi-\lambda_n I)^{-1}$. Hence $\Phi$ is invertible, a~contradiction.

Now for each $n$ find a~unit vector $x_n\in X$ with $\|(\Phi-\lambda_n I)^{-1}x_n\|\to\infty$. Then
$$
\Phi \Bigl(\frac{(\Phi-\lambda_n I)^{-1}x_n}{\|(\Phi-\lambda_n I)^{-1}x_n\|} \Bigr)
=\frac{x_n}{\|(\Phi-\lambda_n I)^{-1}x_n\|}+\frac{\lambda_n(\Phi-\lambda_n I)^{-1}x_n}{\|(\Phi-\lambda_n I)^{-1}x_n\|}\rightarrow 0
$$
as $n\to\infty$. So $0\in\sigma_{ap}(\Phi)$. Hence $\partial\sigma(\Phi)\subset \sigma_{ap}(\Phi)$.

Furthermore,
$$
\partial\sigma(\Phi)=
\partial\sigma(\Phi^*)\subset \sigma_{ap}(\Phi^*)=\sigma_{sur}(\Phi).
$$
\end{proof}

For antilinear operators we have more information. All types of spectrum are circularly symmetric.

\begin{theorem}\label{T4.4}
Let $A$ be an antilinear operator on a~Banach space $X$. Then:
\begin{itemize}
\item[(i)] if $\lambda\in\sigma_{p}(A)$ then $\{\lambda e^{i\theta}: 0\le \theta<2\pi\}\subset\sigma_{p}(A)$;
\item[(ii)] if $\lambda\in\sigma_{ap}(A)$ then $\{\lambda e^{i\theta}: 0\le \theta<2\pi\}\subset\sigma_{ap}(A)$;
\item[(iii)] if $\lambda\in\sigma_{sur}(A)$ then $\{\lambda e^{i\theta}: 0\le \theta<2\pi\}\subset\sigma_{sur}(A)$;
\item[(iv)] if $\lambda\in\sigma(A)$ then $\{\lambda e^{i\theta}: 0\le \theta<2\pi\}\subset\sigma(A)$.
\end{itemize}
\end{theorem}

\begin{proof}
(i) Let $x\in X$, $\|x\|=1$ and $Ax=\lambda x$. Then
$$
A(e^{-i\theta/2}x)=
e^{i\theta/2} Ax=
\lambda e^{i\theta/2}x=
\lambda e^{i\theta}(e^{-i\theta/2}x).
$$
So $\lambda e^{i\theta}\in\sigma_p(A)$.

The remaining statements can be proved similarly.
\end{proof}

\section{Numerical range}
The numerical range of linear operators is already a~classical notion. The numerical range of antilinear operators was studied in \cite{HL}, \cite{ChHL} and \cite{KM}.

Let $X$ be a~Banach space. We write for short
$$
\Pi(X)=\bigl\{(x,x^*)\in X\times X^*: \|x\|=\|x^*\|=\langle x,x^*\rangle=1\bigr\}.
$$
It is natural to define the numerical range of real linear operators in the following way.

\begin{definition}
Let $\Phi\in \cR(X)$. Define the numerical range of $\Phi$ by
$$
W(\Phi)=\bigl\{\langle \Phi x,x^*\rangle: (x,x^*)\in\Pi(X)\}.
$$
\end{definition}

We will frequently use the following Bishop-Phelps-Bollob\'as theorem, see \cite{BD}, Theorem 16.1, which is an important tool for studying numerical ranges in Banach spaces.

\begin{theorem}\label{T3.1}
Let $X$ be a~Banach space, $\e>0$, $x\in X$, $x^*\in X^*$, $\|x\|\le 1$, $\|x^*\|= 1$ and 
$$
|\langle x,x^*\rangle-1\bigr|<\frac{\e^2}{4}.
$$
Then there exists $(y,y^*)\in\Pi(X)$ such that $\|y-x\|<\e$ and $\|y^*-x^*\|<\e$.
\end{theorem}

The next result generalizes \cite{KM}, Theorem 3.5.

\begin{theorem}
Let $\Phi\in \cR(X)$. Then
$$
W(\Phi)\subset W(\Phi^*)\subset\overline{W(\Phi)}.
$$
If $X$ is reflexive, then $W(\Phi^*)=W(\Phi)$.
\end{theorem}

\begin{proof}
Let $\Phi=T+A$ with $T$ a~linear operator and $A$ antilinear operator. Let $\la\in W(\Phi)$. Then there exist $x\in X$ and $x^*\in X^*$ with $\|x\|=\|x^*\|=\langle x,x^*\rangle=1$ and $\langle\Phi x,x^*\rangle=\la$.
Let $\langle Ax,x^*\rangle=|\langle Ax,x^*\rangle|\cdot e^{i\theta}$ for some $\theta\in[0,2\pi)$.
As usually, we identify $X$ with the corresponding subspace of $X^{**}$.

Let $y=e^{i\theta}x\in X\subset X^{**}$ and $y^*=e^{-i\theta}x^*$. Then $\|y\|=\|y^*\|=\langle y^*,y\rangle=1$ and 
$\langle y,\Phi^*y^*\rangle\in W(\Phi^*)$, where
$$
\langle y,\Phi^*y^*\rangle=
\langle y,T^*y^*\rangle+\langle y,A^*y^*\rangle=
\langle x,T^*x^*\rangle+e^{2i\theta}\langle x,A^*x^*\rangle
$$
$$
=
\langle Tx,x^*\rangle +e^{2i\theta}\overline{\langle Ax,x^*\rangle}=
\langle Tx,x^*\rangle+e^{i\theta}|\langle Ax,x^*\rangle|=
\langle\Phi x,x^*\rangle=\la.
$$
So $\la\in W(\Phi^*)$ and $W(\Phi)\subset W(\Phi^*)$. 

If $X$ is a~reflexive Banach space then we have $W(\Phi^*)=W(\Phi)$ by the symmetry.

Let $\lambda\in W(\Phi^*)$ and $\e>0$. So there exist $x^*\in X^*$ and $x^{**}\in X^{**}$ such that
$\|x^*\|=\|x^{**}\|=\langle x^{**},x^{*}\rangle=1$ and $\lambda=\langle x^{**},\Phi^* x^*\rangle$.

Since the unit ball of $X$ is $w^*$-dense in the unit ball in $X^{**}$ by the Goldstine theorem, there exists $x\in X$ such that $\|x\|\le 1$, 
$\bigl|\langle x,x^*\rangle-1\bigr|<\frac{\e^2}{4}$ and
$\bigl|\langle x,\Phi^*x^*\rangle-\lambda\bigr|<\e$.

By the Bishop-Phelps-Bollobas theorem, there exist $(y,y^*)\in\Pi(X)$ such that $\|y-x\|<\e$ and $\|y^*-x^*\|<\e$. 
Let $\Phi=T+A$ with $T$ linear operator and $A$ antilinear operator. Let
$\langle Ay,y^*\rangle=|\langle Ay,y^*\rangle|\cdot e^{i\theta}$ for some $\theta\in[0,2\pi)$. 
Let $z=e^{i\theta}y$ and $z^*=e^{-i\theta}y^*$. Then
$\langle\Phi z,z^*\rangle\in W(\Phi)$ and
$$
\langle\Phi z,z^*\rangle=
\langle T y,y^*\rangle+e^{-2i\theta}\langle Ay,y^*\rangle=
\langle T y,y^*\rangle+\overline{\langle A y,y^*\rangle}=
\langle y,\Phi^*y^*\rangle.
$$
We have
\begin{align*}
&\bigl|\langle \Phi z,z^*\rangle-\lambda\bigr|
=
\bigl|\langle y,\Phi^* y^*\rangle-\lambda\bigr|\cr
\le&
|\langle  y-x,\Phi^* y^*\rangle|+|\langle  x,\Phi^*(y^*-x^*)\rangle|+\bigl|\langle x,\Phi^* x^*\rangle-\lambda\bigr|
<
2\e\|\Phi^*\|+\e.
\end{align*}
Since $\e>0$ was arbitrary, we conclude that $\lambda\in \overline{W(\Phi)}$.

\end{proof}

The numerical range is related to the spectrum in the usual way.

\begin{theorem}
Let $\Phi\in \cR(X)$. Then:
\begin{itemize}
\item[(i)] $\sigma_p(\Phi)\subset W(\Phi)$;
\item[(ii)] $\sigma(\Phi)\subset \overline{W(\Phi)}$.
\end{itemize}
\end{theorem}

\begin{proof}
(i): Let $\lambda\in\sigma_p(\Phi)$. Then there exists a~unit vector $x\in X$ with $\Phi x=\lambda x$. Let $x^*\in X^*$ be a~unit functional satisfying $\langle x,x^*\rangle =1$. Then
$$
\langle\Phi x,x^*\rangle=\langle \lambda x,x^*\rangle=\lambda\in W(\Phi).
$$
\smallskip

(ii): Let $\lambda\in\sigma_{ap}(\Phi)$. Then there exists a~sequence $(x_n)$ unit vector in $X$ with $(\Phi-\lambda I) x_n\to 0$. Let $x_n^*\in X^*$ be unit functionals satisfying $\langle x_n,x_n^*\rangle =1$. Then $\langle\Phi x_n,x_n^*\rangle\in W(\Phi)$ and
$$
\langle\Phi x_n,x_n^*\rangle=
\langle(\Phi-\lambda I) x_n,x_n^*\rangle+\langle\lambda x_n,x_n^*\rangle\to\lambda.
$$
So $\lambda\in\overline{W(\Phi)}$ and $\sigma_{ap}(\Phi)\subset \overline{W(\Phi)}$.

We have
$$
\sigma_{sur}(\Phi)=\sigma_{ap}(\Phi^*)\subset \overline {W(\Phi^*)}=\overline{W(\Phi)}.
$$
So $\sigma(\Phi)=\sigma_{ap}(\Phi)\cup\sigma_{sur}(\Phi)\subset \overline{W(\Phi)}$.

\end{proof}

If $A:X\to X$ is an antilinear operator, then $W(A)$ is circularly symmetric (i.e., if $\lambda\in W(A)$ then $e^{i\theta}\lambda\in W(A)$ for all $\theta\in[0,2\pi)$). For real linear operators, we have the following relation.

\begin{lemma}\label{circle}
Let $X$ be a~Banach space, let $T:X\to X$ be a~linear operator and $A:X\to X$ an antilinear operator. Let $(x,x^*)\in\Pi(X)$. Then
$$
W(T+A)\supset \bigl\{z\in\CC: |z-\langle Tx,x^*\rangle|= |\langle Ax,x^*\rangle|\bigr\}.
$$
\end{lemma}

\begin{proof}
For each $\theta\in[0,2\pi)$, we have $(e^{i\theta}x, e^{-i\theta} x^*)\in\Pi(X)$. So
$$
\langle(T+A)e^{i\theta}x, e^{-i\theta} x^*\rangle=
\langle Tx, x^*\rangle+ e^{-2i\theta} \langle Ax,x^*\rangle\in W(T+A).
$$
So 
$$
\bigl\{z\in\CC: |z-\langle Tx,x^*\rangle|= |\langle Ax,x^*\rangle|\bigr\}\subset W(T+A).
$$
\end{proof}

We are going to prove a~much stronger statement: if $\dim X\ge 2$ and $\|x\|=\|x^*\|=1=\langle x,x^*\rangle$ then
\begin{equation}\label{ball}
W(T+A)\supset \bigl\{z\in\CC: |z-\langle Tx,x^*\rangle|\le |\langle Ax,x^*\rangle|\bigr\}.
\end{equation}

First we need a~stronger version of Theorem 3.3 in \cite{KM}. In the same time we correct a~gap in the proof of this theorem in \cite{KM}.

\begin{theorem}\label{T5.4}
Let $X$ be a~Banach space such that $\dim X\ge 2$. Let $A:X\to X$ be an antilinear operator. Then $0\in W(A)$.
More generally, if $M_0\subset X$ is a~two-dimensional subspace, then there exist $m\in M_0\vee AM_0$ and $x^*\in X^*$ such that
$\|m\|=1=\|x^*\|=\langle m,x^*\rangle$ and $\langle Am,x^*\rangle=0$.
\end{theorem}

\begin{proof}
Without loss of generality we may assume that $\|A\|\le 1$. Let $M_0\subset X$ be a~two-dimensional subspace and $x,y\in M_0$ linearly independent unit vectors. Let $M_1=M_0\vee AM_0$ and $M_2=M_1\vee AM_1=M_0\vee AM_0\vee A^2M_0$. Clearly $\dim M_1<\infty$ and $\dim M_2<\infty$.

Let 
$$
L=\{\langle Am,x^*\rangle: m\in M_1, x^*\in X^*, \|m\|=1=\|x^*\|=\langle m,x^*\rangle\}
$$
$$
=
\{\langle Am,m^*\rangle: m\in M_1, m^*\in M_2^*, \|m\|=1=\|m^*\|=\langle m,m^*\rangle\}.
$$
Clearly $L$ is a~closed set and $L\subset W(A)$. We show that $0\in L$. 

Suppose on the contrary that $0\notin L$. 
So there exists $c>0$ such that
$L\cap\{z\in \CC: |z|\le c\}=\emptyset$.

Let $\e$ be a~positive number, $\e< c^2/128$. By \cite{AV}, Theorem 4.11 and Remark 4.12, there exists a~continuous conjugate-homogeneous mapping $J:M_1\to M_1^{*}$ such that
$$
\|J(u)\|\le (1+\e)\|u\|
$$
and
$$
\bigl|\langle u,J(u)\rangle - \|u\|^2\bigr|\le \e\|u\|^2
$$
for all $u\in M_1$. Clearly for $u\in M_1$, $\|u\|=1$ we have
$$
\|J(u)\|\ge |\langle u,J(u)\rangle|\ge 1-\e.
$$
Define mapping $\widetilde J:M_1\to M_1^{*}$ by $\widetilde J(u)=J(u)\cdot \frac{\|u\|}{\|J(u)\|}$ for $u\ne 0$ and $\widetilde J(0)=J(0)=0$. Clearly $\widetilde J$ is continuous, conjugate homogeneous, $\|\widetilde J(u)\|=\|u\|$ and 
$\|\widetilde J(u)-J(u)\|\le \e\|u\|$ for all $u\in M_1$. 
For a~unit vector $u\in M_1$ we have
$$
\bigl|\langle u,\widetilde J(u)\rangle-1\bigr|\le
\bigl|\langle u,J(u)\rangle-1\bigr|+\|\widetilde J(u)-J(u)\|\le 2\e<\frac{c^2}{64}.
$$

By the Bishop-Phelps-Bollob\'as theorem, there exist $v\in M_1$ and $v^*\in M_1^{*}$ such that $(v,v^*)\in \Pi(M_1)$, $\|v-y\|< c/4$ and $\|v^*-\widetilde J(y)\|<c/4$. Extend $v^*$ to a~functional $x^*\in X^*$ with $\|x^*\|=\|v^*\|=1$. So
$$
c<|\langle Av,x^*\rangle|\le
|\langle Ay,x^*\rangle|+ |\langle Av-Ay,x^*\rangle|
\le
|\langle Ay,v^*\rangle|+\frac{c}{4}
$$
$$
\le
|\langle Ay,\widetilde J(y)\rangle| + |\langle Ay,v^*-\widetilde J(y)\rangle| +\frac{c}{4}
\le
|\langle Ay,\widetilde J(y)\rangle| + \frac{c}{2}.
$$
So $|\langle Ay,\widetilde J(y)\rangle| >c- c/2=c/2$.

Consider the continuous mappings $F,G:\CC\to\CC$ defined by
$$
F(\lambda)=\bigl\langle A(x+\lambda y), \widetilde J(x+\lambda y)\bigr\rangle
$$
and
$$
G(\lambda)=\bigl\langle A(\lambda y),\widetilde J(\lambda y)\bigr\rangle.
$$
Note that $G(\lambda)=\bar \lambda^2 \langle Ay,\widetilde J(y)\rangle$ and $|G(\lambda)|\ge \frac{c|\lambda|^2}{2}$ for all $\lambda\in \CC$.

We have
$$
|F(\lambda)-G(\lambda)|\le
\bigl|\langle Ax,\widetilde J(\lambda y)\rangle\bigr|+\bigl|\bigl\langle A(x+\lambda y), \widetilde J(x+\lambda y)-\widetilde J(\lambda y)\bigr\rangle\bigr|
$$
$$
\le
|\lambda|+(|\lambda|+1)\bigl\|\widetilde J(x+\lambda y)-\widetilde J(\lambda y)\bigr\|.
$$
By continuity of $\widetilde J$ and the compactness of the unit ball in $M_1$, there exists $\delta>0$ such that
$\|\widetilde J(u)-\widetilde J(v)\|< c/8$ whenever $u,v\in M_1$, $\|u\|\le 1$ and $\|u-v\|<\delta$. So for $|\lambda| $ large ($|\lambda|>\delta^{-1}$) we have
$$
\|\widetilde J(x+\lambda y)-\widetilde J(\lambda y)\|=
|\lambda|\cdot \Bigl\|\widetilde J\Bigl(\frac{x}{\lambda}+y\Bigr)-\widetilde J(y)\Bigr\|\le
\frac{|\lambda|c}{8}.
$$
So
$$
|F(\lambda)-G(\lambda)|\le |\lambda|+(|\lambda|+1)\frac{|\lambda|c}{8}\le \frac{|\lambda|^2c}{4}
$$
for all $\lambda$ with $|\lambda|$ sufficiently large.

Fix $r>0$ sufficiently large and consider continuous curves $f,g:\TT\to\CC$ defined by
$$
f(e^{i\theta})=F(re^{i\theta})
$$
and
$$
g(e^{i\theta})=G(re^{i\theta})
=
r^2e^{-2i\theta}\langle Ay,\widetilde J(y)\rangle.
$$
So $|g(e^{i\theta})|\ge\frac{cr^2}{2}$ and
$$
|g(e^{i\theta})-f(e^{i\theta})|\le\frac{cr^2}{4}.
$$
Thus the curves $f$ and $g$ are homotopic, and so $\ind_f(0)=\ind_g(0) =-2$ (for the definition of index and its basic properties see~\cite{R}, subsections 10.38 - 10.40).
So there exists $\mu\in\CC$, $|\mu|<r$ such that $F(\mu)=0$ (otherwise $f$ would be homotopic with a~constant function and so $\ind_f(0)=0$).

Let $u=\frac{x+\mu y}{\|x+\mu y\|}$. We have $u\in M_0$, $\|u\|=1$, $\|\widetilde J(u)\|= 1$ and 
$$
|\langle u,\widetilde J(u)\rangle -1|\le
|\langle u,J(u)\rangle -1|+\|\widetilde J(u)-J(u)\|\le 2\e.
$$
By the Bishop-Phelps-Bollob\'as theorem, there exist $v\in M_1$, $v^*\in M_1^*$ such that $(v,v^*)\in\Pi(M_1)$, $\|v-u\|<2\sqrt{2\e}$ and $\|v^*-\widetilde J(u)\|< 2\sqrt{2\e}$. Extend $v^*$ to a~functional $x^*\in X^*$ with $\|x^*\|=\|v^*\|=1$. We have
$$
|\langle Av,x^*\rangle| \le
|\langle Av-Au, x^*\rangle|+|\langle Au,x^*\rangle|
$$
$$
\le
\|v-u\|+
\|\langle Au, v^*-\widetilde J(u)\rangle|+|\langle Au, \widetilde J(u)\rangle|\le 4\sqrt{2\e}\le \frac{c}{2}
$$
since 
$$
\langle Au,\widetilde J(u)\rangle=\frac{1}{\|x+\mu y\|^2}\bigl\langle A(x+\mu y),\widetilde J(x+\mu y)\bigr\rangle=\frac{F(\mu)}{\|x+\mu y\|^2}=0.
$$
However, $\langle Av,x^*\rangle\in L$, a~contradiction.

\end{proof}

Now we prove (\ref{ball}) for Banach spaces with a~smooth norm.

Recall that a~norm $\|\cdot\|$ on a~Banach space $X$ is called (G\^ateaux) smooth if for every unit vector $x\in X$ there is the unique support functional $x^*\in X^*$ with $\|x^*\|=\|x\|=1=\langle x,x^*\rangle.$

If $X$ is a~separable Banach space and $\e>0$ then there exists an equivalent smooth norm $|||\cdot|||$ on $X$ such that
$$
|||x|||\le \|x\|\le (1+\e)|||x|||
$$
for all $x\in X$, see e.g. \cite{CFY}.

Note also that if $X$ is a~Banach space with a~smooth norm then the mapping $J:S_X\to S_{X^*}$ defined by $\langle x,J(x)\rangle=1$ is norm-$w^*$ continuous by the Smulian lemma, see e.g. \cite{FHHSPZ}, Lemma 8.4.

\begin{proposition}\label{smooth}
Let $X$ be a~Banach space with smooth norm $\|\cdot\|$, $\dim X\ge 2$, let $T:H\to H$ be a~linear operator and $A:H\to H$ an antilinear operator. Let $x\in X$, $x^*\in X^*$, $\|x\|=1=\|x^*\|=\langle x,x^*\rangle$. Then
$$
W(T+A)\supset \bigl\{z\in\CC: |z-\langle Tx,x^*\rangle|\le |\langle Ax,x^*\rangle|\bigr\}.
$$
Moreover, if $M_0\subset X$ is a~two-dimensional subspace containing $x$ then
$$
\bigl\{z\in\CC: |z-\langle Tx,x^*\rangle|\le |\langle Ax,x^*\rangle|\bigr\}
$$
$$
\subset
\bigl\{\langle (T+A)m,x^*\rangle: m\in M_0\vee TM_0\vee AM_0, x^*\in X^*, \|m\|=\|x^*\|=\langle m,x^*\rangle=1\bigr\}.
$$
\end{proposition}

\begin{proof}
Let $M_0\subset X$ be a~two-dimensional subspace containing $x$.
Let $M_1=M_0\vee TM_0\vee AM_0$ and $M_2=M_1\vee TM_1\vee AM_1$.
Let 
$$
L=\bigl\{\langle (T+A)m,x^*\rangle: m\in M_1: x^*\in X^*, \|m\|=\|x^*\|=\langle m,x^*\rangle=1\bigr\}.
$$
Clearly $L\subset W(T+A)$. 
If $|z-\langle Tx,x^*\rangle|=|\langle Ax,x^*\rangle|$ then $z\in L$ as in Lemma \ref{circle}.
Let $z_0\in\CC$,  $|z_0-\langle Tx,x^*\rangle|< |\langle Ax,x^*\rangle|$. Suppose on the contrary that $z_0\notin L$.

By Theorem \ref{T5.4}, there exist $y\in M_1$ and $y^*\in X^*$ such that $(y,y^*)\in \Pi(X)$ and $\langle Ay,y^*\rangle=0$.

Denote by $S_{M_2}$ the unit sphere in $M_2$. 

Find a~continuous curve $f:[0,1]\to S_{M_1}$ with $f(0)=x$ and $f(1)=y$.

Let $J:S_{M_2}\to S_{M_2^*}$ satisfy $(u,J(u))\in\Pi(M_2)$ for all $u\in S_{M_2}$.
Then $J$ is continuous.

For all $t\in[0,1]$ and $\theta\in[0,2\pi)$ we have $\bigl\langle (T+A)e^{i\theta}f(t),e^{-i\theta}J(f(t))\bigr\rangle\in L\subset W(T+A)$. 

The curves $u_0, u_1:[0,2\pi)\to\CC$ defined by
$$
u_0(\theta)= \bigl\langle (T+A)e^{i\theta}f(0),e^{-i\theta}J(f(0))\bigr\rangle
=\langle Tx,x^*\rangle +e^{-2i\theta}\langle Ax,x^*\rangle
$$
and
$$
u_1(\theta) =\bigl\langle (T+A)e^{i\theta}f(1),e^{-i\theta}J(f(1))\bigr\rangle=
\langle Ty,y^*\rangle
$$
are homotopic. However,
${\rm Ind}_{u_0}(z_0)=-2$ and ${\rm Ind}_{u_1}(z_0)=0$, a~contradiction.
So $z_0\in L\subset W(T+A)$.
\end{proof}

In fact, the assumption that the norm is smooth in the previous theorem can be omitted.

\begin{theorem}\label{T5.6}
Let $X$ be an at least two-dimensional Banach space, let $T:X\to X$ be a~linear operator and $A:X\to X$ an antilinear operator. Let $x\in X$, $x^*\in X^*$ and $\|x\|=\|x^*\|=\langle x,x^*\rangle=1$. Then
$$
W(T+A)\supset \bigl\{z\in\CC: |z-\langle Tx,x^*\rangle|\le |\langle Ax,x^*\rangle|\bigr\}.
$$
\end{theorem}

\begin{proof}
Without loss of generality, we may assume that $X$ is separable. In fact, let $X_0$ be the smallest subspace of $X$ that contains $x$ and is invariant for $T$ and $A$. Clearly $X_0$ is separable and $W((T+A)|X_0)\subset W(T+A)$.

If $|z-\langle Tx,x^*\rangle|= |\langle Ax,x^*\rangle|$ then $z\in W(T+A)$ by Lemma \ref{circle}.
Let $z_0\in\CC$ satisfy
$|z_0-\langle Tx,x^*\rangle|< |\langle Ax,x^*\rangle|$.

Let $M_0\subset X$ be a~two-dimensional subspace containing $x$, let $M_1=M_0\vee TM_0\vee AM_0$, $M_2=M_1\vee TM_1\vee AM_1$ and $M_3=M_2\vee TM_2\vee AM_2$.

Let $(\e_n)$ be a~sequence of positive numbers with $\e_n\to 0$. There exist equivalent smooth norms $|||\cdot|||_n$ on $X$ such that
$$
|||u|||_n\le \|u\|\le (1+\e_n)|||u|||_n
$$
for all $u\in X$.
Denote by the same symbol $|||\cdot|||_n$ the dual norms on $X^*$ defined by
$$
|||u^*|||_n=\sup\{|\langle u,u^*\rangle|:u\in X, |||u|||_n\le 1\}.
$$
Let $x_n=\frac{x}{|||x|||_n}$ and $x_n^*=\frac{x^*}{|||x^*|||_n}$. Clearly $x_n\to x$ and $x_n^*\to x^*$.
So $\langle x_n,x_n^*\rangle\to \langle x,x^*\rangle=1$.

By the Bishop-Phelps-Bollob\'as theorem, there exist vectors $u_n\in M_1$ and functionals $u_n^*\in X^*$ with $|||u_n|||_n=|||u_n^*|||_n=1=\langle u_n,u_n^*\rangle$, $|||u_n-x_n|||_n\to 0$ and $|||(u_n^*-x_n^*)|M_1|||_n\to 0$. So $u_n\to x$ and $u_n^*|M_1\to x^*|M_1$.
We have $Tu_n\to Tx\in M_1$ and 
$$
\lim_{n\to\infty}\langle Tu_n,u_n^*\rangle=
\lim_{n\to\infty}\langle Tx,u_n^*\rangle=
\lim_{n\to\infty}\langle Tx, x_n^*\rangle=\langle Tx,x^*\rangle.
$$
Similarly
$\langle Au_n,u_n^*\rangle\to\langle Ax,x^*\rangle$.
So for all $n$ large enough we have
$$
|z_0-\langle Tu_n,u_n^*\rangle|< |\langle Au_n,u_n^*\rangle|.
$$
By Proposition \ref{smooth}, there exist vectors $v_n\in M_2$ and $v_n^*\in X^*$ such that $|||v_n|||_n=1=|||v_n^*|||=\langle v_n,v_n^*\rangle$ such that $\langle (T+A)v_n,v_n^*\rangle=z_0$.
Without loss of generality, we may assume that the sequence $(v_n)$ is convergent, $v_n\to v$. Clearly $v\in M_2, \|v\|=1$.
Similarly, we may assume that
$v_n^*|M_3\to v^*$ for some $v^*\in M_3^*$. So
$$
\|v^*\|=\lim_{n\to\infty}\|v_n^*|M_3\|\le 1
$$
and
$$
\langle v,v^*\rangle=
\lim_{n\to\infty}\langle v, v_n^*\rangle=
\lim_{n\to\infty}\langle v_n,v_n^*\rangle=1.
$$
So $\|v^*\|=1$ and
$$
\langle (T+A)v,v^*\rangle=
\lim_{n\to\infty}\langle (T+A)v,v_n^*\rangle=
\lim_{n\to\infty}\langle (T+A)v_n,v_n^*\rangle=
z_0\in W(T+A).
$$
\end{proof}

\begin{corollary}\label{C2}
Let $\dim X\ge 2$, let $T:X\to X$ be a~linear operator and $A:X\to X$ an antilinear operator. Then
$$
W(T+A)=
\bigcup_{(x,x^*)\in\Pi(X)}
\bigl\{z\in\CC: |z-\langle Tx,x^*\rangle|\le |\langle Ax,x^*\rangle|\bigr\}.
$$
\end{corollary}

\begin{corollary}
Let $\dim X\ge 2$, let $T:X\to X$ be a~linear operator and $A:X\to X$ an antilinear operator, let $\alpha,\beta\in\CC$, $|\alpha|\le|\beta|$. Then $W(T+\alpha A)\subset W(T+\beta A)$. In particular, $W(T)\subset W(T+A)$.
\end{corollary}

\begin{proof}
By Corollary~\ref{C2},
\begin{align*}
W(T+\alpha A)=
&\bigcup_{(x,x^*)\in\Pi(X)}
\bigl\{z\in\CC: |z-\langle Tx,x^*\rangle|\le |\alpha|\cdot|\langle Ax,x^*\rangle|\bigr\}
\\
\subset
&\bigcup_{(x,x^*)\in\Pi(X)}
\bigl\{z\in\CC: |z-\langle Tx,x^*\rangle|\le |\beta|\cdot|\langle Ax,x^*\rangle|\bigr\}
=W(T+\beta A).
\end{align*}
In particular, $W(T)\subset W(T+A)$.
\end{proof}

\section{Numerical range for real linear operators on Hilbert spaces}

In the following $H$ will be a~complex Hilbert space. 
As usually, we identify the dual $H^*$ with $H$. 
This Hilbert space convention causes small changes in some statements. For example, for real linear operators on Hilbert spaces we have
$$
\sigma(\Phi^*)=\{\bar z: z\in\sigma(\Phi)\}.
$$
Similarly,
$$
\sigma_{ap}(\Phi^*)=\{\bar z: z\in\sigma_{sur}(\Phi)\}
$$
and
$$
\sigma_{sur}(\Phi^*)=\{\bar z: z\in\sigma_{ap}(\Phi)\}.
$$

The numerical range of a~real linear operator is defined in the following way.

\begin{definition}
Let $\Phi:H\to H$ be a~real linear operator on a~Hilbert space $H$. Define its numerical range $W(\Phi)$ by
$$
W(\Phi)=\{\langle \Phi x,x\rangle: x\in H, \|x\|=1\}
$$
and the numerical radius $w(\Phi)$ by
$$
w(\phi)=\sup\{|\lambda|: \lambda\in W(\Phi)\}
=\sup\{|\langle \Phi x,x\rangle|: x\in X, \|x\|=1\}.
$$
\end{definition}

The following result is a~consequence of Theorem \ref{T5.6} for Banach spaces. However, the proof for real operators on Hilbert spaces is much simpler.

\begin{theorem}\label{L1}
Let $\dim H\ge 2$, let $T:H\to H$ be a~linear operator and $A:H\to H$ an antilinear operator. Let $x\in H$, $\|x\|=1$. Then
$$
W(T+A)\supset \bigl\{z\in\CC: |z-\langle Tx,x\rangle|\le |\langle Ax,x\rangle|\bigr\}.
$$
\end{theorem}

\begin{proof}
The statement is clear if $\langle Ax,x\rangle=0$.

Suppose that $\langle Ax,x\rangle\ne 0$. By \cite{KM}, there exists a~unit vector $y\in H$ with $\langle Ay,y\rangle=0$.
Find a~continuous curve $f:[0,1]\to \{u\in H: \|u\|=1\}$ with $f(0)=x$ and $f(1)=y$.

For every $\theta\in[0,2\pi)$ we have
$$
\langle Tx,x\rangle +e^{i\theta}\langle Ax,x\rangle=
\bigl\langle (T+A)e^{-i\theta/2}x,e^{-i\theta/2}x\bigr\rangle \in W(T+A).
$$
So if $z\in\CC$, $|z-\langle Tx,x\rangle|=|\langle Ax,x\rangle|$, then $z\in W(T+A)$. 

Let $z_0\in\CC$, $|z_0-\langle Tx,x\rangle|<|\langle Ax,x\rangle|$
Suppose on the contrary that $z_0\notin W(T+A)$.
Let
$$
M=\bigl\{\langle (T+A)e^{i\theta}f(t),e^{i\theta}f(t)\rangle: t\in[0,1],0\le\theta<2\pi\bigr\}.
$$
Clearly $M\subset W(T+A)$. So $z_0\notin M$. 

The curves $u_0, u_1:[0,2\pi)\to\CC$ defined by
$$
u_0(\theta)= \bigl\langle (T+A)e^{i\theta}f(0),e^{i\theta}f(0)\bigr\rangle
=\langle Tx,x\rangle +e^{-2i\theta}\langle Ax,x\rangle
$$
and
$$
u_1(\theta) =\bigl\langle (T+A)e^{i\theta}f(1),e^{i\theta}f(1)\bigr\rangle=
\langle Ty,y\rangle
$$
are homotopic. However,
${\rm Ind}_{u_0}(z_0)=-2$ and ${\rm Ind}_{u_1}(z_0)=0$, a~contradiction.
So $z_0\in M\subset W(T+A)$.
\end{proof}

As in the Banach space setting, this implies the following corollary.

\begin{corollary}\label{C4}
Let $\dim H\ge 2$, let $T:H\to H$ be a~linear operator and $A:H\to H$ an antilinear operator. Then:
\begin{itemize}
\item[(i)]
$$
W(T+A)=
\bigcup_{x\in X\atop\|x\|=1}
\bigl\{z\in\CC: |z-\langle Tx,x\rangle|\le |\langle Ax,x\rangle|\bigr\}.
$$
\item[(ii)]
if $\alpha,\beta\in\CC$, $|\alpha|\le|\beta|$, then $W(T+\alpha A)\subset W(T+\beta A)$. In particular, $W(T)\subset W(T+A)$.
\end{itemize}
\end{corollary}

\begin{proof}
(i) is clear.
\smallskip

(ii) 
By the previous lemma,
\begin{align*}
W(T+\alpha A)=
&\bigcup_{x\in X\atop\|x\|=1}
\bigl\{z\in\CC: |z-\langle Tx,x\rangle|\le |\alpha|\cdot|\langle Ax,x\rangle|\bigr\}
\\
\subset
&\bigcup_{x\in X\atop\|x\|=1}
\bigl\{z\in\CC: |z-\langle Tx,x\rangle|\le |\beta|\cdot|\langle Ax,x\rangle|\bigr\}
=W(T+\beta A).
\end{align*}
In particular, $W(T)\subset W(T+A)$.
\end{proof}

By the classical result of Hausdorff and Toeplitz, the numerical range of a~linear operator is always a~convex subset of the complex plane. By \cite{KM}, the numerical range of an antilinear operator on a~Hilbert space of dimension at least two is also a~convex set (even a~disc).

The numerical ranges of real linear operators may be more complicated sets than those of linear (or antilinear) operators. Nevertheless, we show that the numerical range of a~real linear operator on a~space of dimension at least two is still convex. So this generalizes the Hausdorff-Toeplitz theorem for linear operators.

\begin{theorem}
\label{theo:convexity}
Let $\dim H\ge 2$, let $\Phi:H\to H$ be a~real linear operator. Then $W(\Phi)$ is convex.
\end{theorem}

\begin{proof}
Let $\Phi=T+A$, where $T:H\to H$ is a~linear operator and $A:H\to H$ an antilinear operator.
Suppose on the contrary that $W(T+A)$ is not convex. 

Without loss of generality, we may assume that $\dim H=2$. Indeed, let $x,y\in H$ be unit vectors and $\lambda\in{\rm conv}\,\{\langle (T+A)x,x\rangle, \langle (T+A)y,y\rangle\}$, $\lambda\notin W(T+A)$. By Theorem \ref{L1}, the vectors $x$ and $y$ are linearly independent.
Let $H_0=\bigvee\{x,y\}$. So $\dim H_0=2$.
Let $P$ be the orthogonal projection onto $H_0$. Let $T_0=PT|H_0$ and $A_0=PA|H_0$. Then $T_0,A_0:H_0\to H_0$, $T_0$ is linear, $A_0$ antilinear, $\langle (T+A)x,x\rangle, \langle (T+A)y,y\rangle\in W(T_0+A_0)$ and $W(T_0+A_0)\subset W(T+A)$. So $\lambda\notin W(T_0+A_0)$.

In particular, we may assume that $W(T+A)$ is closed.

We may also assume without loss of generality that there exist $w,w'\in\RR$ with
$w<0<w'$, $w,w'\in W(T+A)$ and $0\notin (T+A)$. Indeed, replace $T+A$ by $\alpha(T+\beta I+A)$ for suitable $\alpha,\beta\in\CC$, $|\alpha|=1$ if necessary.

Moreover, we may assume that $(w,w')\cap W(T+A)=\emptyset$.
There exist unit vectors $v,v'\in H$ with $\langle (T+A)v,v\rangle=w$ and $\langle (T+A)v',v'\rangle=w'$.
Let $\langle Tv,v\rangle=h$ and $\langle Tv',v'\rangle=h'$.
So
$\langle Av,v\rangle=w-h$ and $\langle Av',v'\rangle =w'-h'$.

\bigskip
\noindent{\bf Claim.} $\Re h\le w$ and $\Re h'\ge w'$.
\medskip

\begin{proof}
Suppose on the contrary that $\Re h > w$. Let $\e>0$, $w+\e<\min\{w', \Re h\}$. Then
$$
\bigl|\langle Tv,v\rangle -(w+\e)\bigr|^2
=|\Im h|^2+|\Re h-w-\e|^2
<|\Im h|^2+|\Re h-w|^2
=|\langle Av,v\rangle|^2.
$$
So $w+\e\in W(T+A)$ by Theorem \ref{L1}, a~contradiction with the assumption that $(w,w')\cap W(T+A)=\emptyset$.

Hence $\Re h\le w$. The inequality $\Re h'\ge w'$ can be proved in a~similar way.
\end{proof}

In particular, the previous claim implies that $h'\ne h$. So the vectors $v,v'$ are linearly independent. 

Let $c=\langle (T+A)v,v'\rangle+\langle (T+A)v',v\rangle$. We distinguish three cases:
\medskip
\begin{itemize}
\item[(i)] Let $\Im c=0$. Consider the function $f:[0,2]\to H$ defined by
\begin{align*}
&f(t)=v+tv'\qquad(0\le t\le 1),
\\
&f(t)=(2-t)v+v'\qquad(1\le t\le 2).
\end{align*}
Clearly $f$ is a~continuous function, $f(0)=v$ and $f(2)=v'$. For $t\in[0,1]$ we have
$$
\bigl\langle (T+A)f(t),f(t)\bigr\rangle =
\langle (T+A)v,v\rangle +t^2\langle (T+A)v',v'\rangle +tc
=w+ t^2w'+tc
$$
and 
$$
\Im\bigl\langle (T+A)f(t),f(t)\bigr\rangle =t\Im c=0.
$$
For $t\in[1,2]$ we have
\begin{align*}
\bigl\langle (T+A)f(t),f(t)\bigr\rangle &=
(2-t)^2\langle (T+A)v,v\rangle +\langle (T+A)v',v'\rangle +(2-t)c\\
&=
(2-t)^2w+w'+(2-t)c
\end{align*}
and
$$
\Im \bigl\langle (T+A)f(t),f(t)\bigr\rangle =(2-t)\Im c=0.
$$
So $\bigl\langle (T+A)f(t),f(t)\bigr\rangle \in\RR$ for all $t\in[0,2]$. Moreover,
$\langle (T+A)f(0),f(0)\rangle=w<0$ and $\langle (T+A)f(2),f(2)\rangle=w'>0$. 
By continuity of the function $t\mapsto \Re\bigl\langle (T+A)f(t),f(t)\bigr\rangle $, there exists $t_0\in(0,2)$ with $\Re \bigl\langle (T+A)f(t_0),f(t_0)\bigr\rangle =0$. So $\bigl\langle (T+A)f(t_0),f(t_0)\bigr\rangle =0$ and
$$
0=\Bigl\langle (T+A)\frac{f(t_0)}{\|f(t_0)\|},\frac{f(t_0)}{\|f(t_0)\|}\Bigr\rangle\in W(T+A),
$$
a contradiction.

\item[(ii)] Let $\Im c>0$. Let $f:[0,2]\to H$ be the function considered above. As above, we have $\Im \bigl\langle (T+A)f(t),f(t)\bigr\rangle >0$ for all $t\in(0,2)$.

We have $\Re\langle Tf(0),f(0)\rangle =\Re h\le w<0$ and $\Re\langle Tf(2),f(2)\rangle =\Re h'\ge w'>0$. 
By continuity of the function $t\mapsto \Re\bigl\langle Tf(t),f(t)\bigr\rangle $, there exists $t_1\in(0,2)$ with $\Re \bigl\langle Tf(t_1),f(t_1)\bigr\rangle =0$. By the assumption $0\notin W(T+A)$. So by Theorem \ref{L1}, 
$$
\bigl|\bigl\langle Af(t_1),f(t_1)\bigr\rangle\bigr|<
\bigl|\bigl\langle Tf(t_1),f(t_1)\bigr\rangle\bigr|= \bigl|\Im \bigl\langle Tf(t_1),f(t_1)\bigr\rangle\bigr|.
$$
Since $\Im\bigl\langle (T+A)f(t_1),f(t_1)\bigr\rangle>0$, we have
$\Im\bigl\langle Tf(t_1),f(t_1)\bigr\rangle>0$.

Let $a:=\frac{\bigl\langle Tf(t_1),f(t_1)\bigr\rangle}{\|f(t_1)\|^2}$. So $a\in W(T)\cap i\RR$ with $\Im a>0$.

Consider now the function $g:[0,2]\to H$ defined by
\begin{align*}
&g(t)=v-tv'\qquad(0\le t\le 1),
\\
&g(t)=(2-t)v-v'\qquad(1\le t\le 2).
\end{align*}
Clearly $g$ is a~continuous function. For $t\in[0,1]$ we have
$$
\bigl\langle (T+A)g(t),g(t)\bigr\rangle =
\langle (T+A)v,v\rangle +t^2\langle (T+A)v',v'\rangle -tc
=w+ t^2w'-tc
$$
and 
$$
\Im\bigl\langle (T+A)g(t),g(t)\bigr\rangle =-t\Im c<0.
$$
For $t\in[1,2]$ we have
\begin{align*}
\bigl\langle (T+A)g(t),g(t)\bigr\rangle &=
(2-t)^2\langle (T+A)v,v\rangle +\langle (T+A)v',v'\rangle -(2-t)c\\
&=
(2-t)^2w+w'-(2-t)c
\end{align*}
and
$$
\Im \bigl\langle (T+A)g(t),g(t)\bigr\rangle =-(2-t)\Im c<0.
$$
Again there exists $t_2\in(0,2)$ with $\Re \langle Tg(t_2),g(t_2)\rangle =0$. So by Theorem \ref{L1}, 
$$
\bigl|\langle Ag(t_2),g(t_2)\rangle\bigr|<
\bigl|\langle Tg(t_2),g(t_2)\rangle\bigr|= \bigl|\Im Tg(t_2),g(t_2)\rangle\bigr|.
$$
Since $\Im\bigl\langle (T+A)g(t_2),g(t_2)\bigr\rangle<0$, we have
$\Im\bigl\langle Tg(t_2),g(t_2)\bigr\rangle<0$.

Let $b:=\frac{\bigl\langle Tg(t_2),g(t_2)\bigr\rangle}{\|g(t_2)\|^2}$. So $b\in W(T)\cap i\RR$ with $\Im b<0$.
By the convexity of the set $W(T)$, we conclude that $0\in W(T)\subset W(T+A)$, a~contradiction.

\item[(iii)] Let $\Im c<0$. This case can be treated similarly. The functions $f$ and $g$ considered above will imply that there are $a,b\in W(T)\cap i\RR$ with $\Im a<0$ and $\Im b>0$. So $0\in W(T)\subset W(T+A)$, which again is a~contradiction.
\end{itemize}
Hence the set $W(T+A)$ is convex.
\end{proof}

\begin{remark}
The assumption $\dim H\ge 2$ is clearly necessary. Let $\Phi_{\alpha,\beta}$ be a circlet (this class of real linear operators was introduced and studied in~\cite{HN}), i.e. the real linear operator acting on Hilbert space $H=\CC$ given by
$\Phi_{\alpha,\beta}z=\alpha z+\beta \bar{z}\quad(\alpha,\beta,z\in\CC)$, which is the most general form of real linear operator on $\CC$. It is easy to check that $W(A)=\{\alpha+|\beta|z\in\CC: |z|=1\}$, which is convex if and only if the antilinear part of $\Phi_{\alpha,\beta}$ vanishes.
\end{remark}

One of the basic properties of the numerical range of linear operators $T$ is inequality
\begin{equation}
\label{eq:1} 
\frac{\|T\|}{2}\leq w(T)\leq \|T\|.
\end{equation}
Moreover, $T=\Re T +i\Im T$, where $\Re T=\frac{T+T^*}{2}$ and $\Im T=\frac{T-T^*}{2i}$ are self-adjoint operators. Each self-adjoint operator $S$ satisfies $W(S)\subset \RR$ and $w(S)=\|S\|$.

If $A:H\to H$ is an antilinear operator, then one can write $A=A_1+A_2$, where
$$
A_1=\frac{A+A^*}{2}\qquad{\rm and}\qquad A_2=\frac{A-A^*}{2}
$$
satisfy $A_1^*=A_1$ and $A_2^*=-A_2$.
By~\cite{KM}, $W(A_2)=\{0\}$, and so $W(A)=W(A_1)$.

\begin{theorem}\label{T6.6}
Let $\Phi:H\to H$ be a~real linear operator, $\Phi^*=\Phi$. Then $w(\Phi)=\|\Phi\|$.
\end{theorem}

\begin{proof}
Clearly $w(\Phi)\le\|\Phi\|$.

Let $\Phi=T+A$, where $T$ is a~linear operator and $A$ is an antilinear operator. We have $T^*=T$ and $A^*=A$.

Let $x,y\in H$ be unit vectors. We have
$$
\langle A(x+y),x+y\rangle-\langle A(x-y),x-y\rangle =2\langle Ax,y\rangle+2\langle Ay,x\rangle=
4\langle Ax,y\rangle
$$
and
$$
\langle T(x+y),x+y\rangle-\langle T(x-y),x-y\rangle =2\langle Tx,y\rangle+2\langle Ty,x\rangle=
4\Re\langle Tx,y\rangle.
$$
So
\begin{align*}
&4\bigl|\Re\bigl\langle (T+A)x,y\bigr\rangle\bigr|=
\Bigl|\Re\bigl\langle(T+A)(x+y),x+y\bigr\rangle -\Re\bigl\langle(T+A)(x-y),x-y\bigr\rangle\Bigr|
\cr
\le
&\bigl|\bigl\langle(T+A)(x+y),x+y\bigr\rangle\bigr| + \bigl|\bigl\langle(T+A)(x-y),x-y\bigr\rangle\bigr|
\cr
\le
&w(T+A)\bigl(\|x+y\|^2+\|x-y\}|^2\bigr)
=
2w(T+A)\bigl(\|x\|^2+\|y\|^2\bigr)=4w(T+A)
\end{align*}
by the parallelogram law. 
Hence
$$
\|T+A\|=\sup\bigl\{|\Re \langle(T+A)x,y\rangle|:\|x\|=\|y\|=1\bigr\}\le w(T+A),
$$
and so $w(\Phi)=\|\Phi\|$.
\end{proof}

\begin{corollary}
\label{corol}
Let $A:H\to H$ be an antilinear operator. Then
$$
W(A)=W\Bigl(\frac{A+A^*}{2}\Bigr)\qquad {\rm and} \qquad w(A)=\frac{\|A+A^*\|}{2}.
$$
\end{corollary}

The Corollary~\ref{corol} was first proven in~\cite{JL} using the theory of complex symmetric operators. Here, we present an alternative approach.

For a~real linear operator $\Phi$ on a~Hilbert space $H$, let $r(\Phi)$ denote the spectral radius of $\Phi$, $r(\Phi)=\max\{|\lambda|: \lambda\in\sigma(\Phi)\}$. Clearly $r(\Phi)\le w(\Phi)\le\|\Phi\|$.

\begin{theorem}\label{T6.8}
Let $\Phi$ be a~real linear operator on a~Hilbert space $H$. If $w(\Phi)=\|\Phi\|$ then
$$
r(\Phi)=w(\Phi)=\|\Phi\|.
$$
In particular, if $\Phi^*=\Phi$ then $r(\Phi)=w(\Phi)=\|\Phi\|$.
\end{theorem}

\begin{proof}
Without loss of generality, we may assume that $\|\Phi\|=1$. If $w(\Phi)=\|\Phi\|=1$ then there exists a~sequence $(x_n)$ of unit vectors in $H$ such that $|\langle \Phi x_n,x_n\rangle|\to 1$. So there exist $\lambda\in\CC$, $|\lambda|=1$ and a~subsequence $(x_{n_k})$ such that
$$
\lim_{k\to\infty}\langle \Phi x_{n_k},x_{n_k}\rangle=\lambda.
$$
We have
$$
\|\Phi x_{n_k}-\lambda x_{n_k}\|^2=
\|\Phi x_{n_k}\|^2+\|\lambda x_{n_k}\|^2-2\Re\langle \Phi x_{n_k},\lambda x_{n_k}\rangle\le
2-2\Re\bar\lambda\langle \Phi x_{n_k}, x_{n_k}\rangle\to 0.
$$
Hence $\lambda\in\sigma_{ap}(\Phi)$ and $r(\Phi)=1$.
The last part of the theorem is a straightforward consequence of the obtained result and Theorem~\ref{T6.6}.
\end{proof}

\begin{theorem}
Let $A:H\to H$ be an antilinear operator and $A_1=(A+A^*)/2$ its self-adjoint part. Then the following conditions are equivalent:
\begin{itemize}
    \item[(i)] $w(A)=\|A\|$;
    \item[(ii)] $\|A\|=\|A_1\|$;
    \item[(iii)] $\{z\in\CC: |z|=\|A\|\}\subset \sigma_{ap}(A)$;
    \item[(iv)] $\{z\in\CC: |z|=\|A\|\}\subset \sigma_{ap}(A_1)$.
\end{itemize}
\end{theorem}

\begin{proof}
By Corollary \ref{corol}, $w(A)=w(A_1)=\|A_1\|$. So, (i)$\Leftrightarrow$(ii).
		
(i)$\Rightarrow$ (iii): By Theorem \ref{T6.8}, $r(A)=\|A\|$. By Theorems \ref{T4.4} and \ref{T4.3} (v), 
$\{z\in\CC: |z|=\|A\|\}\subset \sigma_{ap}(A)$.

(iii)$\Rightarrow$ (i): Clearly, (iii) implies that $r(A)=\|A\|$, and so $w(A)=\|A\|$.

(ii)$\Leftrightarrow$ (iv): We have $w(A_1)=\|A_1\|$. By Theorem \ref{T6.8}, $r(A_1)=\|A_1\|$. By Theorems \ref{T4.4} and \ref{T4.3} (v), $\{z\in\CC: |z|=\|A_1\|\}\subset \sigma_{ap}(A_1)$.  So (ii) implies (iv).

Conversely, (iv) implies that $\|A\|\le r(A_1)\le\|A_1\|\le\|A\|$, and so $\|A\|=\|A_1\|$.
\end{proof}

For an antilinear operator $A$ we have either $W(A)=\{z\in\CC: |z|\le\|A_1\|\}$ or  $W(A)=\{z\in\CC: |z|<\|A_1\|\}$.
In~\cite{JL} it was proved that $W(A)$ is a closed disc if and only if the operator $A_1$ is norm-attaining.
The next result gives an equivalent criterion.

\begin{theorem}
Let $A:H\to H$ be an antilinear operator and $A_1=(A+A^*)/2$ its self-adjoint part. Then
$\|A_1\|\in \sigma_{p}(A_1)$ if and only if $W(A)=\{z\in\CC: |z|\le \|A_1\|\}$.
\end{theorem}

\begin{proof}
The statement is clear if $A_1=0$. So without loss of generality, we may assume that $\|A_1\|=1$.

If $1\in\sigma_p(A_1)$ then there exists a~unit vector $x\in H$ with $A_1x=x$. Then 
$$
1=\langle A_1x,x\rangle\in W(A_1)=W(A),
$$
and so $\{z\in\CC: |z|\le 1\}\subset W(A)$.
Hence $W(A)=\{z\in\CC: |z|\le 1\}$.

Conversely, if $x\in H$ is a~unit vector satisfying $\langle Ax,x\rangle=1=\|A_1\|$, then
$\langle A_1x,x\rangle=1$ and
$$
\|A_1x-x\|^2=
\|A_1x\|^2+\|x\|^2-2\Re\langle A_1x,x\rangle=
\|A_1x\|^2-1\le 0.
$$
So $A_1x=x$ and $1\in\sigma_p(A_1)$.

\end{proof}
The following result is a generalization of inequality~(\ref{eq:1}) to the case of real linear operators on complex Hilbert spaces. 
It is also a significant improvement of Proposition~$6.14$ from~\cite{HN}.
\begin{theorem}
Let $T:H\to H$ be a~linear operator, $A:H\to H$ antilinear, let $A_1=(A+A^*)/2$. Then $W(T+A)=W(T+A_1)$ and
$$
\frac{1}{2}\|T+A_1\|\le
\frac{1}{2}\|T+2A_1\|\le
w(T+A_1)=w(T+A)\le \|T+A_1\|.
$$
\end{theorem}

\begin{proof}
Let $A_2=\frac{A-A^*}{2}$. For $x\in H$, $\|x\|=1$ we have 
$$
\langle (T+A)x,x\rangle=\langle (T+A_1)x,x\rangle+\langle A_2x,x\rangle=\langle (T+A_1)x,x\rangle.
$$
So $W(T+A)=W(T+A_1)$.

Let $x,y\in H$ be unit vectors. We have
$$
\langle A_1(x+y),x+y\rangle-\langle A_1(x-y),x-y\rangle =2\langle A_1x,y\rangle+2\langle A_1y,x\rangle=
4\langle A_1x,y\rangle
$$
since $A_1^*=A_1$.
So
$$
\frac{1}{4}\Bigl(\langle A_1(x+y),x+y\rangle-\langle A_1(x-y),x-y\rangle +i\langle A_1(x+iy),x+iy\rangle-i\langle A_1(x-iy),x-iy\rangle \Bigr)
$$
$$
=
\frac{1}{4}\bigl(4\langle A_1x,y\rangle+4i\langle A_1 x,iy\rangle\bigr)
=2\langle A_1x,y\rangle.
$$
Furthermore,
$$
\frac{1}{4}\Bigl(\langle T(x+y),x+y\rangle-\langle T(x-y),x-y\rangle +i\langle T(x+iy),x+iy\rangle-i\langle T(x-iy),x-iy\rangle \Bigr)=\langle Tx,y\rangle
$$
by the polar formula. So
\begin{align*}
&\bigl|\langle (T+2A_1)x,y\rangle\bigr|
\\
\le
&\frac{1}{4}\Bigl|\Bigl(\langle (T+A_1)(x+y),x+y\rangle-\langle (T+A_1)(x-y),x-y\rangle +i\langle (T+A_1)(x+iy),x+iy\rangle
\\
&\hskip 2cm -i\langle (T+A_1)(x-iy),x-iy\rangle \Bigr)\Bigr|
\\
\le
&\frac{1}{4}w(T+A_1)\bigl(\|x+y\|^2+\|x-y\|^2+\|x+iy\|^2+\|x-iy\|^2\bigr)
\\
=
&\frac{1}{4}w(T+A_1)\bigl(2\|x\|^2+2\|y\|^2+2\|x\|^2+2\|iy\|^2\bigr)
\\
=
&w(T+A_1)(\|x\|^2+\|y\|^2)=2w(T+A_1)
\end{align*}
by the parallelogram law. 
Hence 
$$
w(T+A_1)\ge \frac{1}{2}\sup\left\{|\langle (T+2A_1)x,y\rangle|: \|x\|=\|y\|=1\right\}=\frac{1}{2}\|T+2A_1\|
$$ 
and the other inequalities are clear.
\end{proof}

The Hilbert space convention also makes slight changes in the duality properties of the numerical range.

As in the Banach spaces setting, if $\Phi\in \cR(X)$ is such that $\Phi=T+A$ with $T$ linear operator and $A$ antilinear operator, we define $\Phi^*:H\to H$ by $\Phi^*=T^*+A^*$.
where $T^*:H\to H$ is the dual of $T$, i.e., the linear operator satisfying $\langle Tx,y\rangle=\langle x,T^*y\rangle$ for all $x, y\in X$ and $A^*:H\to H$ the antilinear dual of $A$, i.e., the antilinear operator satisfying $\langle Ax,y\rangle=\overline{\langle x,A^*y\rangle}$ for all $x,y\in H$. 

The duality satisfies the usual properties as in the Banach space setting.

If $T:H\to H$ is a~linear operator then 
$W(T^*)=\bigl\{\bar z: z\in W(T)\bigr\}$.
If $A:H\to H$ is an antilinear operator, then $W(A^*)=W(A)$. Since $W(A)$ is a~disc (if $\dim H\ge 2$) or a~circle  (if $\dim H=1$) centered at the origin, we also have $W(A^*)=\bigl\{\bar z: z\in W(A)\bigr\}$).

The same relation holds also for general real linear operators.

\begin{proposition} Let $\Phi:H\to H$ be a~real linear operator. Then 
$$
W(\Phi^*)=\bigl\{\bar z: z\in W(\Phi)\bigr\}.
$$
\end{proposition}

\begin{proof}
Let $\Phi=T+A$ with $T$ a~linear operator and $A$ an antilinear operator. Let $\la\in W(\Phi)$. 
Let $x\in H$ be a~unit vector such that $\langle \Phi x,x\rangle=\la$. Let $\langle Ax,x\rangle=|\langle Ax,x\rangle|\cdot e^{i\theta}$ for some $\theta\in[0,2\pi)$. We have
$$
\langle \Phi^* x,x\rangle =
\langle T^*x,x\rangle +\langle A^*x,x\rangle
=
\overline{\langle Tx,x\rangle}+\langle Ax,x\rangle
$$
$$
=\overline{\langle Tx,x\rangle +\overline{\langle Ax,x\rangle}},
$$
where
$$
\langle Tx,x\rangle +\overline{\langle Ax,x\rangle}
=\langle (T+A)e^{i\theta}x,e^{i\theta}x\rangle
\in W(\Phi).
$$
So $W(\Phi^*)\subset \bigl\{\bar z: z\in W(\Phi)\bigr\}$. The equality follows from the symmetry.
\end{proof}

It is well known that the numerical range of a~linear operator on a~two-dimensional space is always an elliptical disc (possibly degenerated to a~line segment or a~single point). The numerical range of an antilinear operator on a~two-dimensional space is even a~ball.
The next example shows that, for general real linear operators, the situation is more complicated.

\begin{example}
Let $e_1,e_2$ be the standard basis in $\CC^2$. Define $\Phi: \CC^2\to \CC^2$ by
$\Phi (\alpha e_1+\beta e_2)= \alpha e_1+ \frac{\bar\beta}{2} e_2$. Then the linear/antilinear decomposition of $\Phi=T+A$ is formed by
$$
T(\alpha e_1+\beta e_2)=\alpha e_1
$$
and
$$
A(\alpha e_1+\beta e_2)=\frac{\bar\beta  e_2}{2}.
$$
By Corollary \ref{C4},
$$
W(\Phi)=\bigcup_{|\alpha|^2+|\beta|^2=1}\{z\in\CC: |z-|\alpha|^2|\le |\beta|^2/2\}
$$
$$
=\bigcup_{0\le t\le 1}\Bigl\{z\in\CC: |z-t|\le\frac{1-t}{2}\Bigr\}.
$$
It is easy to verify that $W(\Phi)$ is the convex hull of the ball
$\{z: |z|\le 1/2\}$ and the point $1$, which is not an elliptical disc.
\end{example}

The next example is similar, but more general.

\begin{example} Let $H=\CC^2$ be a~Hilbert space and $\Phi:\CC^2\rightarrow\CC^2$ a~real linear operator defined by
\begin{equation*}
    \Phi\left(\begin{array}{l}
         x_1  \\
         x_2
    \end{array}\right)
    =
   \left(\begin{array}{l}
         \lambda_1 x_1+\eta_1\bar{x}_1  \\
         \lambda_2 x_2+\eta_2\bar{x}_2
    \end{array}\right),\quad x_1,x_2\in\CC,
\end{equation*}
where $\lambda_1,\lambda_2,\eta_1,\eta_2\in\CC$. Then $W(\Phi)={\rm conv}\left(\left(\lambda_1+\overline{\mathbb{D}}_{|\eta_1|}\right)\cup\left(\lambda_2+\overline{\mathbb{D}}_{|\eta_2|}\right)\right)$, where ${\overline{\mathbb D}}_{\gamma}=\{z\in\CC:|z|\leq\gamma\}$, for any $\gamma\geq 0$, is a closed disc or radius~$\gamma$. 

\end{example}

\begin{proof}

Let $\mathcal{A}_1,\mathcal{A}_2$ be subsets of $\CC$ defined as
\begin{equation*}
    \mathcal{A}_k=\left\{z\in\CC: |z-\lambda_k|\leq |\eta_k|\right\}=\lambda_k+|\eta_k|\overline{\mathbb{D}},\quad k\in\{1,2\}.
\end{equation*}
For the unit sphere $S_{\CC^2}$ in $\CC^2$ we have
\begin{equation*}
    S_{\CC^2}=\left\{\left(\begin{array}{c}
         re^{i\alpha}  \\
         \sqrt{1-r^2}e^{i\beta} 
    \end{array}\right)\in\CC^2: \alpha,\beta\in\RR,r\in [0,1]\right\},
\end{equation*}
and
\begin{align*}
    \left<\Phi\left(\begin{array}{c}
         re^{i\alpha}  \\
         \sqrt{1-r^2}e^{i\beta} 
    \end{array}\right),
    \left(\begin{array}{c}
         re^{i\alpha}  \\
         \sqrt{1-r^2}e^{i\beta} 
    \end{array}\right)\right>
    &=
    \left<\left(\begin{array}{c}
         r\left(\lambda_1e^{i\alpha}+\eta_1 e^{-i\alpha}\right)  \\
         \sqrt{1-r^2}\left(\lambda_2e^{i\beta}+\eta_2 e^{-i\beta}\right)
    \end{array}\right),
    \left(\begin{array}{c}
         re^{i\alpha}  \\
         \sqrt{1-r^2}e^{i\beta} 
    \end{array}\right)\right>\\
    &=
    \lambda_1 r^2+\eta_1 r^2 e^{-2i\alpha}+\lambda_2 (1-r^2)+\eta_2 (1-r^2) e^{-2i\beta}.
\end{align*}
Substituting $\alpha'=\textrm{arg}(\eta_1)-2\alpha$ and $\beta'=\textrm{arg}(\eta_2)-2\beta$, we obtain
\begin{equation*}
    W(\Phi)=\left\{r^2\left(\lambda_1+|\eta_1| e^{i\alpha'}\right)+(1-r^2)\left(\lambda_2+|\eta_2|e^{i\beta'}\right)\in\CC:\alpha',\beta'\in\RR, r\in[0,1]\right\}.
\end{equation*}
Fixing $r=1$, we obtain
\begin{equation*}
    \left\{\lambda_1+|\eta_1| e^{i\alpha'}\in\CC:\alpha'\in\RR\right\}=\partial\mathcal{A}_1\subset W(\Phi)
\end{equation*}
and similarly fixing $r=0$, we obtain
\begin{equation*}
    \left\{\lambda_2+|\eta_2| e^{i\beta'}\in\CC:\beta'\in\RR\right\}=\partial\mathcal{A}_2\subset W(\Phi).
\end{equation*}
We have
\begin{equation*}
    \textrm{conv}(\mathcal{A}_1\cup\mathcal{A}_2)
		=\left\{t\left(\lambda_1+\rho_1|\eta_1|e^{i\varphi_1}\right)+(1-t)\left(\lambda_2+\rho_2|\eta_2|e^{i\varphi_2}\right): \varphi_1,\varphi_2\in\mathbb{R},t,\rho_1,\rho_2\in [0,1]\right\}.
\end{equation*}
Taking $t=r^2, \rho_1=1,\rho_2=1,\varphi_1=\alpha',\varphi_2=\beta'$, we notice that $W(\Phi)\subset\textrm{conv}\left(\mathcal{A}_1\cup\mathcal{A}_2\right)$.
As $\partial\mathcal{A}_1\cup \partial\mathcal{A}_2\subset W(\Phi)\subset\textrm{conv}\left(\mathcal{A}_1\cup\mathcal{A}_2\right)$, from the convexity of $W(\Phi)$ we conclude that $W(\Phi)=\textrm{conv}\left(\mathcal{A}_1\cup\mathcal{A}_2\right)$.
\end{proof}

\section{Numerical results}
Calculating the numerical range of a~real linear operator is a~difficult task in general, even for the Hilbert space $H=\CC^2$. Hence, numerical simulations are useful for this problem. In~\cite{J} Johnson presented an algorithm to numerically calculate the boundary of the numerical range of linear operators acting on the Hilbert space $\CC^n$, $n\in\NN$, but to the best of our knowledge, there is no generalization of this algorithm to the case of real linear operators. It seems that the most reliable way to numerically calculate the numerical range of a~real linear operator~$\Phi$ (possessing canonical linear/antilinear decomposition $\Phi=T+A$, where $T$ is linear operator and $A$ is antilinear operator) acting on Hilbert space $H=\CC^n$, $n\in\NN$, is to first sample vectors $\left\{x_1,x_2,...,x_N\right\}$ from uniform distribution on $S_{\CC^n}$ for some $N\in\NN$. Then from Corollary~\ref{C4} and Theorem~\ref{theo:convexity}, we know that
\begin{equation*}
    \textrm{conv}\left( \bigcup_{k=1}^N \left\{z\in\CC: \left|z-\left<Tx_k,x_k\right>\right|=\left|\left<Ax_k,x_k\right>\right|\right\}\right)\subset W(T+A),
\end{equation*}
and for $N$ large enough, we expect the left-hand side of the equation above to closely resemble the $W(T+A)$.
To generate vectors from uniform distribution on $\CC^n$ it is enough to sample vectors from uniform distribution on $S_{\RR^{2n}}$ as it was described e.g. in~\cite{M} and then transform them with unitary mapping $U:\RR^{2n}\rightarrow \CC^n$, $U(x_1,...,x_{2n})=(x_1+ix_2,...,x_{2n-1}+ix_{2n})$.

Now we will present two examples of real linear operators acting on Hilbert space $\CC^2$ such that their numerical ranges are clearly neither elliptical discs nor convex hulls of the unions of two discs.

\begin{example}
Let $H=\CC^2$ be a~Hilbert space and $\Phi_1:\CC^2\rightarrow\CC^2$ be the real linear operator defined by
\begin{equation*}
    \Phi_1\left(\begin{array}{l}
         x_1  \\
         x_2
    \end{array}\right)
    =
   \left(\begin{array}{c}
         x_1+\bar{x}_1  \\
         -x_2+\bar{x}_1
    \end{array}\right),\quad x_1,x_2\in\CC.
\end{equation*}
Its numerical range is presented on Figure~\ref{fig:1} and $10^3$ points were sampled to generate it. Increasing the number of sampled points beyond $10^3$ has no visible effect on the shape of the numerical range.
\begin{figure}[H]
\label{fig:1}
 \begin{center}
  \includegraphics[width=10cm]{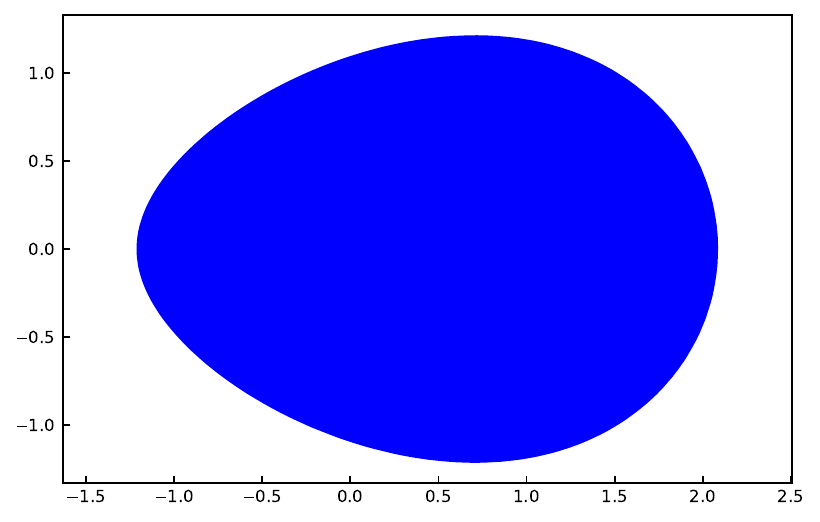}
  \caption{Numerical range of real linear operator $\Phi_1.$}
 \end{center}
\end{figure}
\end{example}

\begin{example}
Let $H=\CC^2$ be a~Hilbert space and $\Phi_2:\CC^2\rightarrow\CC^2$ be the real linear operator defined by
\begin{equation*}
    \Phi_2\left(\begin{array}{l}
         x_1  \\
         x_2
    \end{array}\right)
    =
   \left(\begin{array}{c}
         x_2  \\
         -x_1+x_2+\bar{x}_2
    \end{array}\right),\quad x_1,x_2\in\CC,
\end{equation*}
Its numerical range is presented on Figure~\ref{fig:2} and $10^3$ points were sampled to draw it. Increasing the number of sampled points beyond $10^3$ has no visible effect on the shape of the numerical range.
\begin{figure}[H]
\label{fig:2}
 \begin{center}
  \includegraphics[width=10cm]{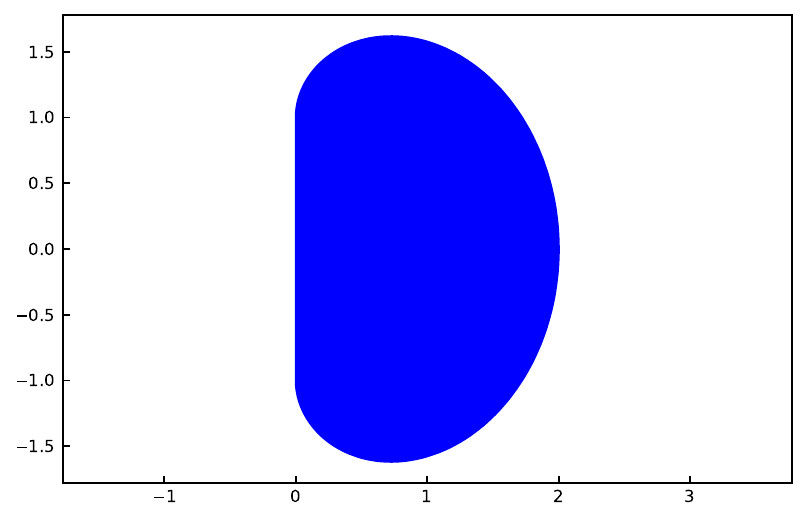}
  \caption{Numerical range of real linear operator $\Phi_2.$}
 \end{center}
\end{figure}
From the Figure~2 we may suppose that the boundary of $W(\Phi_2)$ contains line segment. It is possible to confirm this surmise by direct calculations.
Unit sphere in Hilbert space $\CC^2$ can be parametrized as follows 
\begin{equation*}
    S_{\CC^2}=\left\{u(r,\alpha,\beta)\in\CC^2: \alpha,\beta\in\RR,r\in [0,1]\right\},
\end{equation*}
where 
\begin{equation*}
    u(r,\alpha,\beta)=\left(\begin{array}{c}
         re^{i\alpha}  \\
         \sqrt{1-r^2}e^{i\beta} 
    \end{array}\right).
\end{equation*}
We have
\begin{equation*}
    \left<\Phi_2 u(r,\alpha,\beta),u(r,\alpha,\beta)\right>=2\left(1-r^2\right)\cos^2 \beta+i\left(2r\sqrt{1-r^2}\sin\left(\beta-\alpha\right)+(1-r^2)\sin 2\beta\right),
\end{equation*}
hence it is straightforward to notice that $\textrm{min}\left(\Re W(\Phi_2)\right)=0$ and
\begin{equation*}
    \left\{\Im \left<\Phi_2 x,x\right>\in \RR: x\in S_{\CC^2}, \Re \left<\Phi_2 x,x\right>=0\right\}
    =
    [-1,1],
\end{equation*}
which implies that $\{it\in\CC: t\in [-1,1]\}\subset \partial W(\Phi_2).$ 
\end{example}

\end{document}